\documentclass[11pt,twoside]{amsart}
\usepackage{amsfonts,amssymb,amsmath}
\usepackage{xcolor}
\usepackage[pagebackref]{hyperref}
\usepackage{float}
\makeatletter  
\@namedef{subjclassname@2020}{%
\textup{2020} Mathematics Subject Classification}
\makeatother

\numberwithin{equation}{section}
\newtheorem{theorem}{Theorem}[section]
\newtheorem{lemma}[theorem]{Lemma}
\newtheorem{proposition}[theorem]{Proposition}

\newtheorem{corollary}[theorem]{Corollary}

\theoremstyle{definition}
\newtheorem{definition}[theorem]{Definition}
\newtheorem{remark}[theorem]{Remark}
\newtheorem{example}[theorem]{Example}
\newtheorem{procedure}[theorem]{Procedure}

\def\M{\mathcal{M}}
\def\X{\mathcal{X}}

\begin{document}

\title[On the regularity index of the minimum distance function]{On
the regularity index of the minimum distance function in projective  
nested Cartesian codes}

\thanks{The first author was partially supported by grants from CNPq (PQ 
308708/2023-7)
and FAPEMIG (APQ-01430-24). The second author
was partially supported by the Center for Mathematical Analysis, Geometry and
Dynamical Systems of Instituto Superior T\'ecnico, Universidade de
Lisboa. The third author was partially supported by SNII, M\'exico.}

\author[C. Carvalho]{C\'{\i}cero  Carvalho}
\address{
Faculdade de Matem\'atica \\ Universidade Federal de Uberl\^andia 
Av.\ J.\ N.\ \'Avila 2121, 38.408-902 - Uberl\^andia - MG, Brazil.
}
\email{cicero@ufu.br}

\author[M. Vaz Pinto]{Maria Vaz Pinto}
\address{Departamento de Matem\'atica, Instituto Superior T\'ecnico,
Universidade de Lisboa, Avenida Rovisco Pais, 1, 1049-001 Lisboa,
Portugal.
}
\email{vazpinto@math.tecnico.ulisboa.pt}

\author[R. H. Villarreal]{Rafael H. Villarreal}
\address{
Departamento de
Matem\'aticas\\
Cinvestav, Av. IPN 2508, 07360, CDMX, M\'exico.
}
\email{rvillarreal@cinvestav.mx}


\begin{abstract}
Let $\X$ be a projective nested product of fields 
and let $\delta_\X(d)$ be the minimum distance in degree $d\geq 1$ of
the projective nested Cartesian code $C_\X(d)$. The regularity index 
${\rm reg}(\delta_\X)$ of the minimum distance function $\delta_\X$
is the minimum 
integer $d_0\geq 0$ 
such that $\delta_\X(d)=1$ for $d\geq d_0$. We give a 
formula for ${\rm reg}(\delta_\X)$ by determining 
an indicator function of least degree for each point of $\X$ and using
the fact that ${\rm reg}(\delta_\X)$ is the ${\rm v}$-number of the
vanishing ideal $I_\X$ of $\X$. Then we give an
arithmetical criterion that characterizes when $\X$ is
Cayley--Bacharach. 
\end{abstract}
  
\maketitle
\section{Introduction}\label{section-intro}
Let $S:=K[t_0,\ldots,t_n]=\bigoplus_{d=0}^\infty S_d$ be a 
polynomial ring with the standard grading over a finite field
$K=\mathbb{F}_q$ with $q$ elements, let $\prec$ be the 
graded lexicographic order $\prec$ on $S$, where $t_0 \prec \cdots
\prec t_n$, and let $\mathbb{P}^{n}$ be the projective space over the field 
$K$.
Monomials of $S$ are abbreviated
 by $t^\alpha:=t_0^{\alpha_0}\cdots t_n^{\alpha_n}$, for 
 $\alpha=(\alpha_0,\dots,\alpha_n)$ in
 $\mathbb{N}^{n+1}$.

Given a tower $K_0\subset\cdots\subset K_n$ of subfields of $K$, 
let $\X$ be the {\it projective nested product of fields\/}
\begin{eqnarray*}
\mathcal{X}:=\left[K_0\times\cdots\times 
K_n\right]:=\left\{(a_0:\cdots :a_n) \in \mathbb{P}^{n} : a_i\in K_i
\mbox{ for all } i\right\}, 
\end{eqnarray*}
that was introduced in \cite{cnl2017}, and let $d_i$ be the
cardinality of $K_i$. 
Denote by $P_1,\ldots,P_m$ the points of $\mathcal{X}$ written in standard 
representation, that is, for each $P_i$,  
the first nonzero entry from the left is equal to $1$. 
The  evaluation map given by   
\begin{equation*}
\varphi_d\colon S_{d}\longrightarrow K^{|\mathcal{X}|},\ \ \ \ \ 
f\mapsto \left(f(P_1),\ldots,f(P_m)\right),
\end{equation*}
is a linear map of 
$K$-vector spaces. The image of $\varphi_d$, denoted by $C_{\mathcal{X}}(d)$,
defines a linear code that is called a \textit{projective nested Cartesian code} of
degree $d$. This family of codes, introduced and studied by
Carvalho, Neumann and L\'{o}pez \cite{cnl2017}, generalizes the
classical projective Reed--Muller codes studied by S{\o}rensen 
\cite{sorensen} that are obtained by letting
$K_i=\mathbb{F}_q$ for all $i$.  

The \textit{minimum distance} of $C_\X(d)$, denoted $\delta_\X(d)$, is
given by 
$$\delta_\X(d):=\min\{\|v\|
\colon 0\neq v\in C_\X(d)\},$$
where $\|v\|$ is the number of non-zero
entries of $v$. The asymptotic behavior of this function 
is known and according to \cite[Proposition~2.15]{hilbert-min-dis},
there is an 
integer $d_0\geq 0$ such that
$$
|\X|=\delta_\X(0)>\delta_\X(1)>\delta_\X(2)>
\cdots>\delta_\X(d_0)=\delta_\X(d)=1\ \mbox{ for all
}\ d\geq d_0. 
$$
\quad The \textit{regularity index} of 
the \textit{minimum distance function} $\delta_\X$, $d\mapsto\delta_\X(d)$,
$d\geq 0$, denoted ${\rm reg}(\delta_\X)$, is given by
$$
{\rm reg}(\delta_\X):=\min\{d\geq 0 :
\delta_\X(d)=1\}.
$$
\quad Potentially good codes,  
capable of correcting errors in the transmission of
information, should have minimum distance greater than $1$
because the code $C_\X(d)$ can correct up to $t$ errors, where
$$
t=\left\lfloor\frac{\delta_\X(d)-1}{2}\right \rfloor,
$$ 
see \cite[pp.~40--41]{Huffman-Pless}, \cite{MacWilliams-Sloane}.
This is a reason for looking for 
algorithms and formulas to compute the regularity index of 
$\delta_\X$.     

We use indicator functions to determine the
regularity index of 
$\delta_\X$. An indicator function for a point $P\in\X$ is a
homogeneous polynomial 
$f\in S$ such that 
$f(P) \neq 0$ and $f(Q) = 0$ for all $Q \in \X \setminus \{ P \}$. 
These functions are used in coding
theory \cite{min-dis-generalized,dual,sorensen}, Cayley--Bacharach 
schemes \cite{geramita-cayley-bacharach,Guardo-Marino-Van-Tuyl,tohaneanu-vantuyl}, 
and interpolation problems \cite{cocoa-book}. An indicator function of
$P\in\X$ can be computed using \cite[Corollary~6.3.11]{cocoa-book}, see also
\cite{Ceria-etal} and references therein.

Let $I_\X$ be the graded \textit{vanishing ideal} of
$\X$ generated by the homogeneous polynomials in $S$ that
vanish at all points of $\X$. For each point $P$ of $\X$, let 
$\mathfrak{p}=I_P$ be the vanishing ideal of $P$, and let 
${\rm Ass}(I_\X)=\{I_P : P\in \X\}$ be the 
set of \textit{associated primes} of 
$I_\X$. Then, $I_\X=\bigcap_{P\in \X}I_P$ is the primary
decomposition of $I_\X$. 

Following \cite{min-dis-generalized}, we define the v-number of $I_\X$ locally at each point
$P$ of $\X$\/ as:
\begin{equation*}
{\rm v}_{P}(I_\X):=\mbox{min}\{d\geq 0 : \exists\, f\in S_d
\mbox{ with }(I_\X\colon f)=I_P\},
\end{equation*}
where $(I_\X\colon f):=\{g\in S\, :  gf\subset I_\X \}$ is a
\textit{colon ideal}. The least degree of an indicator function of $P$
is ${\rm v}_{P}(I_\X)$ \cite[Lemma~3.2(b)]{villa2024}, and this number can be computed using 
\textit{Macaulay}$2$ \cite{mac2} and the following algebraic
description of the v-number of $I_\X$ at $P$ 
\cite[Proposition~4.2]{min-dis-generalized}:
\begin{equation*}
{\rm
v}_{P}(I_\X)=\alpha\!\left((I_\X\colon I_P)/{I_\X}\right),
\end{equation*}
where $\alpha\!\left((I_\X\colon I_P)/{I_\X}\right)$ is the 
minimum degree of the non-zero elements of
$(I_\X\colon I_P)/{I_\X}$. 
Local v-numbers appear in the work of Geramita, Kreuzer and Robbiano.
More precisely the degree $\deg_\X(P)$ of
a point $P$, in the
sense of \cite[Definition~2.1]{geramita-cayley-bacharach}, is equal to
${\rm v}_{P}(I_\X)$. 

The monomial order $\prec$ we choose comes into play  by considering 
the \textit{footprint} $\Delta(I_\X)$ of the ideal $I_\X$ which is the set of monomials 
$t^\alpha$ of $S$ which are not a leading monomial of any non-zero
polynomials in $I_\X$. The monomials of $\Delta(I_\X)$ are the
\textit{standard monomials} of $I_\X$ and an indicator function $f$ of
a point $P\in \X$ is called \textit{standard} if all the monomials that appear
in $f$ are standard.

For each $P\in\X$ there exists a unique, up to multiplication by a scalar from
$K^*$, standard indicator function $f$ of $P$ of degree ${\rm
v}_P(I_\X)$, and $(I_X\colon I_P)/I_\X$ is a principal ideal of
$S/I_\X$ generated by $\overline{f}=f+I_\X$ \cite[Proposition~3.4]{villa2024}. This gives an algebraic
method to compute $f$ \cite[Procedure~A.1]{villa2024} using
\textit{Macaulay}$2$  \cite{mac2}. 
In
Proposition~\ref{ind-funct}, we give an explicit formula for an
indicator function $f$ of $P$ and we show in Theorems~\ref{p-ej} and
\ref{cicero-maria-vila-ei} that $f$ is of degree ${\rm v}_P(I_X)$,
$f + I_\X$ generates $(I_X\colon I_P)/I_\X$ but $f$ is not necessarily standard. 

The v-{\em number} of $I_\X$, denoted ${\rm
v}(I_\X)$, is the following algebraic invariant of
$I_\X$ that was introduced by Cooper et.\ al.\
\cite{min-dis-generalized} to study the asymptotic
behavior of $\delta_\X$:
\begin{equation*}
{\rm v}(I_\X):=\min\{d\geq 0  : \exists\, f 
\in S_d \mbox{ and }\mathfrak{p} \in {\rm Ass}(I_\X) \mbox{ with } 
(I_\X\colon f)
=\mathfrak{p}\}.
\end{equation*}
\quad The v-number of    
$I_\X$ is the least degree among all indicator function
of the points of $\X$. By \cite[Corollary~5.6]{min-dis-generalized},
we have 
$$
{\rm v}(I_\X)={\rm reg}(\delta_\X),
$$
that is, $\delta_\X(d)=1$ if and only if $d\geq {\rm v}(I_\X)$. 

There is another fundamental invariant related to indicator functions
and v-numbers. The {\it Hilbert function} of $S/I$, $I=I_\X$, is given
by: 
$$
H_\X(d):=\dim_K(S_d/I_d),\ \ \ d=0,1,2,\ldots,
$$
where $I_d=I\cap S_d$. The dimension of $C_\X(d)$ is equal to
$H_\X(d)$ for all $d\geq 0$. There is an integer $r\geq 0$ such that  
$$
1=H_\X(0)<H_\X(1)<\cdots<H_\X(r-1)<H_\X(d)=|\X|\
\mbox{ for all }d\geq r,
$$ 
$r={\rm reg}(H_\X)$ is the \textit{regularity index} of $S/I_\X$, $r$
is also ${\rm reg}(S/I_\X)$, the Castelnuovo--Mumford regularity of
$S/I_\X$, and $r$ is also the degree of the $h^*$-polynomial 
of $S/I(\X)$, see \cite[Remark~1.1]{geramita-cayley-bacharach} and 
\cite[Proposition~9.3.13]{monalg-3rd-edition}. The \textit{degree} or
\textit{multiplicity} of $S/I_\X$ is $|\X|$
\cite{geramita-cayley-bacharach,hilbert-min-dis}. A formula for $|\X|$
is given in \cite[Theorem~2.8]{cnl2017}, see Eq.~\eqref{apr19-26}.
These invariants are related as follows:
$$
\min\{{\rm v}_P(I_\X) : P\in \X\}=
{\rm v}(I_\X)={\rm reg}(\delta_\X)\leq {\rm reg}(H_\X)=\max\{{\rm
v}_P(I_\X) : P\in \X\}=\sum_{i = 1 }^n (d_i - 1) + 1. 
$$
\quad The last two equalities follow from
\cite[Lemma~3.2(d)]{villa2024} and
Lemma~\ref{lemma1.1}. The inequality follows from 
Lemma~\ref{lemma1.2}. There are examples where the inequality is 
strict (Example~\ref{ex1}).

Two of the main results of this paper show that for each point $P\in\X$, we
can give an explicit formula of an indicator function $f$ of $P$ of
degree ${\rm v}_{P}(I_\X)$ and a formula for ${\rm v}_{P}(I_\X)$ 
(Theorems~\ref{p-ej} and
\ref{cicero-maria-vila-ei}). To prove these results we use standard
indicator functions and a result of Carvalho \cite{car-2013} that allows us to
prove a vanishing criterion (Lemma~\ref{zero-function}) that is used
to show Theorem~\ref{cicero-maria-vila-ei}.

We come to our third main result.

\noindent \textbf{Theorem~\ref{cicero-maria-vila-reg}}\textit{ 
Let $\X=[K_0\times\cdots\times K_n]\subset\mathbb{P}^n$ be a
projective nested product of fields, let $d_i=|K_i|$, and let 
${\rm reg}(\delta_\X)$ be the regularity index of $\delta_X$. The following
hold.
\begin{enumerate}
\item[\rm(a)] ${\rm reg}(\delta_\X)=m_n(d_n - 1)+1$, where
$m_n=\min\{m : m\in\mathbb{N},\ m(d_n-1)>\sum_{i=1}^{n-1}(d_i-1)\}$. 
\item[\rm(b)] ${\rm v}_P(I_\X)\in\{{\rm v}_{e_0}(I_\X),\ldots,{\rm
v}_{e_n}(I_\X)\}$
for all $P\in \X$, where $e_j$ is the $j$-th unit vector of
$\mathbb{P}^n$.
\item[\rm(c)] $1+\sum_{i=1}^n(d_i-1)={\rm reg}(H_\X)={\rm
v}_{e_0}(I_\X)\geq\cdots\geq{\rm
v}_{e_n}(I_\X)=m_n(d_n-1)+1={\rm reg}(\delta_\X).
$
\end{enumerate}
}

We say that the finite set $\X\subset \mathbb{P}^n$ is \textit{Cayley-Bacharach} 
if every hypersurface of degree less than
${\rm reg}(S/I_\X)={\rm reg}(H_\X)$ which contains all but one
point of $\X$ must contain all points of $\X$ or
equivalently  ${\rm v}_P(I_\X)={\rm reg}(H_\X)$ for
all $P\in \X$ \cite[Definition~2.7]{geramita-cayley-bacharach}. As an
application, we give an arithmetical criterion that 
characterizes when $\X$ is Cayley--Bacharach (Corollary~\ref{coro1}).

In Appendix~\ref{procedure-degrees}, we give a procedure 
for \textit{Macaulay}$2$ \cite{mac2}, based on
Theorem~\ref{cicero-maria-vila-reg}, to obtain all 
possible values of ${\rm v}_P(I_\X)$ for $P\in\X$, and to compute 
the standard indicator function of degree ${\rm v}_P(I_\X)$ for
$P\in\X$ (Examples~\ref{ex2} and \ref{ex1},
Procedure~\ref{procedure1}). 

\section{Preliminaries}
Let $S:=K[t_0,\ldots,t_n]=\bigoplus_{d=0}^\infty S_d$ 
be a polynomial ring  
over a field $K$ with the standard grading. We will always assume that the set $\M$ of
monomials of $S$ 
is endowed with a 
monomial order $\prec$.

\begin{definition} Let $I \subset S$ be an ideal. The footprint of $I$ is the 
set 
\[
\Delta(I) = \{M \in \M : M \textrm{ is not a leading monomial of any\ } f \in 
I, f \neq 0\}. 
\]
For any integer $d \geq 0$ we define
\[
\Delta_d(I) = \{M \in \Delta(I) :   \deg(M) = d\}. 
\]
\end{definition}

The image of $\Delta(I)$, under the canonical 
map $S\mapsto S/I$, $x\mapsto \overline{x}$, is a basis of $S/I$ as a
$K$-vector space \cite[Proposition~6.52]{Becker-Weispfenning}. This
is a classical 
result of Buchberger and Macaulay (see \cite[Chapter~5]{CLO}). In
particular, the Hilbert function $H_I(d):=\dim_K(S_d/I_d)$ is the number of standard
monomials of degree $d$ for all $d\geq 0$, that is, $H_I(d)=|\Delta_d(I)|$ for 
all $d\geq 0$.

Let $\mathbb{P}^{n}$ be the $n$-dimensional projective space over the field $K$,
and let $\X \subset 
\mathbb{P}^{n}$.
Denote by $P_1,\ldots,P_m$ the points of $\mathcal{X}$ written with standard 
representation for
projective points, namely, the first nonzero entry from the left is equal to
1. 
The  evaluation map  
\begin{equation*}\label{ev-map}
\varphi_d\colon S_{d}\longrightarrow K^{|\mathcal{X}|},\ \ \ \ \ 
f\mapsto \left(f(P_1),\ldots,f(P_m)\right),
\end{equation*}
defines a linear map of
$K$-vector spaces. The image of $\varphi_d$, denoted by $C_{\mathcal{X}}(d)$,
defines a linear code .

In this paper we work with the 
following setup, already considered in 
\cite{cnl2017}. Let $K=\mathbb{F}_q$ be a finite field, 
let $K_0,\ldots, K_n$ be 
a collection of non-empty subsets of $K$, and let $\X$ be the 
{\it projective nested Cartesian set\/}
\begin{eqnarray*}
\mathcal{X}:=\left[K_0\times K_1\times\cdots\times 
K_n\right]:=\left\{(a_0:\cdots :a_n) \in \mathbb{P}^{n} : a_i\in K_i
\mbox{ for all } i\right\}.
\end{eqnarray*}
\quad We denote by $d_i$ the
cardinality of $K_i$ for $i=0,\ldots,n$. We shall always 
assume that $2\leq d_i\leq d_{i+1}$ for all $i$. From
\cite[Theorem~2.8]{cnl2017} we get that 
\begin{equation}\label{apr19-26}
| \X | = 1 + \sum_{i = 1}^n d_i \cdots
d_n.
\end{equation}
\quad We will denote by $S=K[t_0, \ldots, t_n]$ the
polynomial ring over the field $K$ and by $I_\X$ the vanishing 
ideal of $\X$. The Hilbert function of $S/I_\X$ is denoted by $H_\X$.

\begin{lemma}\label{lemma1.1}
The regularity of $H_\mathcal{X}$ is equal to $\sum_{i = 1 }^n (d_i - 1) + 
1$.
\end{lemma}
\begin{proof}
Consider the following sets and their corresponding vanishing ideals
\begin{eqnarray*}
\mathcal{X}_{i}:=\left[K_{n-i}\times\cdots\times K_n\right] \text{ and } 
\mathcal{X}_{i}^{*}:=\left[1\times K_{n+1-i}\times\cdots\times K_n\right], \\
\text{ so \ } 
I_{\X_{i}}\subset K[t_{n-i},\ldots,t_n] \text{ and }
I_{\X_{i}^\ast}\subset K[t_{n-i},\ldots,t_n], 
\text{for } i=0,\ldots,n, 
\end{eqnarray*}
where $1$ in the Cartesian product represents the set $\{1\}$. 

From \cite[Lemma 2.6]{cnl2017} we have that 
$H_{\mathcal{X}_n}(d)=H_{\mathcal{X}_{n-1}}(d)+H_{\mathcal{X}_n^*}(d-1)$.
Applying this result recursively, we get that
\begin{equation}\label{apr15-26}
H_\mathcal{X}(d)=H_{\mathcal{X}_0}(d)+
\displaystyle \sum_{j=1}^nH_{\mathcal{X}_j^{*}}(d-1).
\end{equation}
Since $\mathcal{X}_0=[1],$ then $I_{\mathcal{X}_0}=0$ and 
$H_{\mathcal{X}_0}(d)=1$.
For $j \in \{1, \ldots, n\}$ we have that 
$$H_{\mathcal{X}_j^{*}}(d-1) = 
\ ^a H_{\mathcal{Y}_j}(d-1),
$$ 
where $\mathcal{Y}_j = K_{n+1-j}\times\cdots\times 
K_n \subset \mathbb{A}^{j}(K)$, and  
$\ ^a H_{\mathcal{Y}_j}(d-1)$ is the affine Hilbert function. From
\cite[Proposition~2.5 and Lemma 2.8]{lrv2014}, we get that 
$\mathrm{reg}(H_{\mathcal{X}_j^{*}}(d-1)) = \sum_{i = n + 1 - j}^n (d_i - 1)$
for all $j = 1,\ldots, n$ so the value of 
$\sum_{j=1}^nH_{\mathcal{X}_j^{*}}(d-1)$ is constant when $d - 1 
\geq  \sum_{i = 1 }^n (d_i - 1)$ but not before that, and by
Eq.~\eqref{apr15-26} the proof is
complete. 
\end{proof}

We have the following result.

\begin{lemma}\label{lemma1.2}
$\mathrm{reg}(\delta_{\mathcal{X}})  \leq
\mathrm{reg}(H_\mathcal{X})$.
\end{lemma}
\begin{proof}
From \cite[Lemma 3.1]{cnl2017} we get that $\mathrm{reg}
(\delta_{\mathcal{X}})\leq \sum_{i = 1 }^n (d_i - 1) + 
1$, and the result follows from Lemma~\ref{lemma1.1}.
\end{proof}

In what follows we assume that $K_0\subset \cdots \subset K_n$ are subfields of 
$K$, with $|K_i|=d_i$ for all $0 \le i \le n$. Observe that
$d_{i+1} = d_i^{r_i}$, for some $r_i \ge 1$,
in particular $d_i - 1 \mid d_{i + 1} - 1$ for all 
$i = 0, \ldots, n-1$. 
Then, 
$\mathcal{X}=\left[K_0\times\cdots\times K_n\right]$ is a projective nested 
Cartesian set which
is called a {\it projective nested product of fields.} 

From now on we choose 
the graded lexicographic monomial order $\prec$  in $S$,
where $t_0 \prec \cdots \prec t_n$.
Then, from \cite[Proposition~2.11]{cnl2017} we get that
\begin{equation} \label{basedeG}
\mathcal{G} := \left\{ t_i t_j (t_j^{d_j - 1} - t_i^{d_j - 1}) :
0\leq i < j\leq n 
\right\}
\end{equation}
is a Gr\"obner basis for the vanishing ideal $I_\X$ of $\mathcal{X}$.

\section{First results}

Recall that $|K_i| = d_i$, for $i = 0, \ldots, n$ and we consider
$\X=[K_0 \times \cdots \times K_n] \subset \mathbb{P}^n(K_n)$.

\begin{definition}
Let $P \in \X$. An indicator function for $P$ is a homogeneous polynomial $f 
\in S$ such that 
$f(P) \neq 0$ and $f(Q) = 0$ for all $Q \in \X \setminus \{ P \}$.
\end{definition}

\begin{proposition}\label{ind-funct}
Let $P = (a_0 : \cdots : a_n) \in \X$ and let $j \in \{0, \ldots, n\}$ the least integer such 
that $a_j \neq 0$. For $j \in \{2, \ldots, n\}$ let $m_j$ be the least integer such $m_j 
(d_j - 1) > \sum_{i = 1}^{j - 1} (d_i - 1)$. Then $P$ admits an indicator 
function of degree $1 + \sum_{i = 1}^n (d_i - 1)$, if $j = 0$ or $j = 1$, and of degree 
$m_j (d_j - 1) + 1 + \sum_{i = j + 1}^n (d_i - 1)$, if $j \geq 2$,
here we disregard the last sum when $j = n$. An
indicator function of this degree, assuming that  
$a_j = 1$,  is 
\[ 
f = t_0 \prod_{i = 1}^n  \prod_{\lambda \in K_i \setminus \{a_i\}} (t_i - 
\lambda t_0) \;\;\; 
\textrm{if } j = 0; 
\] 
\[ 
f = t_1 (t_1^{d_1 - 1} - t_0^{d_1 - 1}) \prod_{i = 2}^n  \prod_{\lambda \in 
K_i \setminus \{a_i\}} 
(t_i - \lambda t_1) \;\;\; 
\textrm{if } j = 1; 
\] 
and
\[
f = t_j (f_1 - f_2) \; \prod_{\ell = j + 1}^n \prod_{\lambda \in K_\ell 
\setminus\{a_\ell\}} (t_\ell - \lambda
t_j) \;\;\; 
\textrm{if } j \geq 2, 
\]
where
\begin{equation*}
\begin{split}
f_1 =& t_j^{m_j(d_j - 1)} - \sum_{i = 1}^{j - 1} t_i^{m_j(d_j - 1)}  \\ & + 
\sum_{s = 2}^{j - 
1} \sum_{1 \leq j_1 < \cdots < j_s \leq j - 1}
(-1)^s t_{j_1}^{m_j(d_j - 1) - \sum_{i = 2}^s (d_{j_i} \, - \, 1) } \; 
t_{j_2}^{d_{j_2}\, - 
\,1} \cdots t_{j_s}^{d_{j_s} \, - \, 1} 
\end{split} 
\end{equation*}
and 	
\[
f_2 = t_0^{m_j(d_j - 1) - \sum_{i = 1}^{j - 1} (d_{i} \, - \, 1)  } \prod_{i 
= 1}^{j - 1} 
(t_0^{d_{i} \, - \, 1} - t_i^{d_{i} \, - \, 1})
\]
$($note that if $j = n$ then we take $f =  t_n (f_1 - f_2)$\rm{)}.
\end{proposition}
\begin{proof}
The case where $j = 0$ is simple to verify, and we omit the proof. Let $Q = (b_0 : \cdots : 
b_n) \in \X$. If $j = 1$ then we clearly have $f(Q) = 0$ in the cases where $b_0 \neq 0$ and 
$b_1 \neq 
0$, or $b_1 = 0$. If $b_1 \neq 0$ and $b_0 = 0$ then we may assume $b_1 = 1$ and it's easy to 
check that $f(Q) \neq 0$ if and only if $Q = P$. Before proceeding to the case $j \geq 2$ we 
note that, since $K_i \subset K_{i + 1}$ for $i = 0, \ldots, n - 1$, we have that $d_i - 1$ is 
a factor of $d_{i + 1} - 1$ for all $i = 0, \ldots, n - 1$. Assume now  that $j \geq 2$.  We 
note that 
if $b_j = 0$ then $f(Q) = 0$, so we assume from now on that $b_j \neq 0$. Suppose that there 
are 
$u$ elements in the set $\{b_0, \ldots, b_{j - 1}\}$ which are nonzero, with $u \geq 1$. If $u 
= 1$ and $b_0 \neq 0$ then $f_1(Q) = 1$ and $f_2(Q) = 1$, hence $f(Q) = 0$. If $u \geq 1$  and 
$b_0 = 0$ then $f_2(Q) = 0$ and 
\[
f_1(Q) = 1 - u + \sum_{s = 2}^u (-1)^s \binom{u}{s} = \sum_{s = 0}^u (-1)^s \binom{u}{s} = (1 - 
1)^u = 0
\]
so $f(Q) = 0$. If $u \geq 2$ and $b_0 \neq 0$ then, as above,  $f_1(Q) = 0$ and also $f_2(Q) = 
0$ because at least one factor in the product is equal to zero. Finally, we treat the case 
where $b_0 = \cdots = b_{j - 1} = 0$ and we assume $b_j = 1$. We have $f_1(Q) = 1$ and $f_2(Q) 
= 0$, and it's simple to check that $f(Q) \neq 0$ if and only if $Q = P$.
\end{proof}

For each $j \in \{2, \ldots, n\}$ we defined  $m_j$ as the least integer
such that $m_j 
(d_j - 1) > \sum_{i = 1}^{j - 1} (d_i - 1)$, so equivalently $m_j$ is 
the greatest integer such that
\begin{equation} \label{equiv}
(m_j - 1)(d_j - 1) \leq \sum_{i = 1}^{j - 1} (d_i - 1),
\end{equation}
 or the greatest integer such that
\begin{equation} \label{equiv2}
m_j(d_j - 1) \leq \sum_{i = 1}^{j} (d_i - 1).
\end{equation} 
This definition makes sense for $j = 1$, and then 
$m_1 = 1$. Also the formula for the degree of the indicator functions in the above result, in 
the case where $j \geq 2$, also holds for $j = 1$ with $m_1 = 1$. So from now on we assume that 
$m_1 = 1$.

\begin{lemma}\label{min-deg}
The least degree of the indicator functions in Proposition \ref{ind-funct} is $m_n (d_n - 1) + 
1$.
\end{lemma}
\begin{proof}
We know that in the cases $j = 0$ and $j = 1$ the degrees of the
indicator functions in 
Proposition~\ref{ind-funct} coincide, 
being equal to $1 + \sum_{i = 1}^n (d_i - 1)$, and  
from the definition of $m_n$ we get that $(m_n - 1) (d_n - 1) \leq \sum_{i = 
1}^{n - 1} (d_i - 1)$ so $m_n (d_n - 1) + 1 \leq \sum_{i = 
1}^{n} (d_i - 1) + 1$.

Now we  prove that
\[
m_n (d_n - 1) + 1= \min \{ m_j(d_j - 1) + 1 + \sum_{i = j + 1}^n (d_i - 1) 
: j = 2, \ldots, 
n 
\}.
\]
Let $j \in \{ 2,\ldots, n - 1\}$, we recall that $d_j - 1 \mid d_{j + 1} - 
1$ and we note that 
\[
(m_{j + 1} - 1) \left( \frac{d_{j + 1} - 1}{d_{j} - 1} \right) (d_j - 1) = (m_{j + 1} - 1)(d_{j 
+ 1} - 1)
\]
so from \eqref{equiv} we have that 
\[
(m_{j + 1} - 1) \left( \frac{d_{j + 1} - 1}{d_{j} - 1} \right) (d_j - 1) \leq \sum_{i = 1}^{j} 
(d_i - 1).
\]
\quad Letting $k=(m_{j+1}-1)\big(\frac{d_{j+1}-1}{d_j-1}\big)$, one has
$k(d_j-1)\leq\sum_{i=1}^j(d_i-1)$, and from the equivalent definition
 \eqref{equiv2} for $m_j$ we must have that 
$$
k=(m_{j+1}-1)\left(\frac{d_{j+1}-1}{d_j-1}\right)\leq m_j,
$$
so $(m_{j + 1} - 1) (d_{j + 1} - 1) \leq m_j (d_{j} - 1)$. From this we get
$m_{j + 1}  (d_{j + 1} - 1) \leq m_j (d_{j} - 1) + (d_{j + 1} - 1)$ which implies that
\begin{equation}\label{apr16-26-1}
m_{j+1}(d_{j+1} - 1) + 1 + \sum_{i = j + 2}^n (d_i - 1)  \leq m_j(d_j - 1) + 1 + \sum_{i = j + 
1}^n  (d_i - 1)
\end{equation}
and this concludes the proof.
\end{proof}

Now we prove an important property of indicator functions of minimal degree.

\begin{proposition}\label{xj-div-f}
Let $h$ be an indicator function for the point $P = (a_0 : \cdots : a_n) \in 
\X$, such that all the monomials in $h$ belong to $\Delta(I_\X)$. Let $j 
\in \{0, \ldots, n\}$ be an integer such 
that $a_j \neq 0$. Then, $t_j | h$. Furthermore, if $h$ is of minimal
degree, then $t_j^2 \nmid h$.
\end{proposition}
\begin{proof}
We write $h = t_j h_1 + h_2$, where $h_1$ and $h_2$ are homogeneous 
polynomials and $h_2$ does not have the variable $t_j$ (i.e.\ $\deg_{t_j} 
h_2 = 0$). From the hypothesis we get that 
\[
h(a_0, \ldots, a_{j - 1}, 0 , 
a_{j+1}, \ldots, a_n) = 0 \textrm{ for all  } (a_0 : \ldots : a_{j - 1} : 0 
: a_{j+1}: \ldots : a_n) \in \X.
\]
Thus, taking $\X_{\widehat{j}} := [K_0 \times \cdots \times \widehat{K_j} 
\times \cdots 
\times K_n]$, where $\widehat{K_j}$ means that $K_j$ is not a factor in the 
product, we get that 
\[
h_2(a_0, \ldots, a_{j - 1},  
a_{j+1}, \ldots, a_n) = 0 \textrm{ for all  } (a_0 : \ldots : a_{j - 1} : 
a_{j+1} : \ldots : a_n) \in \X_j,  
\]
so $h_2 \in I_{\X_{\widehat{j}}}$. 
Clearly 
$I_{\X_{\widehat{j}}} \subset 
I_{\X}$ and 
since we are 
assuming that all the monomials in $h$ are in $\Delta(I_\X)$ we must have 
$h_2 = 0$. This shows that $t_j \mid h$, and 
 if 
$h$ is of minimal degree 
we must have $t_j^2 \nmid h$,  otherwise $h_1 = t_j \tilde{h}_1$ would be an 
indicator function for $P$ with degree one less than the degree of
$h$. 
\end{proof}

\section{Main result}

\begin{definition} Let $P\in \X$. We say an indicator function $g$
of $P$ is 
\textit{standard} if all
monomials that occur in $g$ are standard. 
\end{definition}

In what follows, we denote by $e_j$ the $j$-th canonical 
basis vector of $K^n$, where $j = 0, \ldots , n$.




\begin{theorem}\label{p-ej}
Let $[e_j] \in \mathbb{P}^n$, $0\leq j<n$, and let
$P = [e_j + \sum_{i=j+1}^n a_i e_{i}]$, $a_i\in K_i$, $a_k\neq 0$ for
some $k$. Then, the least degree of an indicator function for 
$[e_j]$ is the least degree of an indicator function for $P$ and 
${\rm v}_P(I_\X)={\rm
v}_{e_j}(I_\X)$. 
\end{theorem}

\begin{proof} By \cite[Lemma~3.2(b)]{villa2024}, 
${\rm v}_{e_j}(I_\X)$  (resp. ${\rm v}_{P}(I_\X)$) is the least degree
of an indicator function of 
$[e_j]$ (resp. $P$). 
Define the linear map  \(T:S_1\to S_1\) induced by
\[
T(t_i)=
\begin{cases}
t_i, &\mbox{if } i\leq j,\\
t_i-a_i t_j, &\mbox{if } i>j.
\end{cases}
\]
\quad This map preserves $\X$. Indeed, if $Q=(q_0:\cdots:q_n)\in \X$,
then, for $i>j$, $q_i-a_iq_j\in K_i$, because \(q_i\in K_i\),
\(a_i\in K_i\), and \(q_j\in K_j\subset K_i\). Hence \(T(Q)\in \X\). 
Also, \(T(P)=e_j\). The map \(T\) is invertible, with inverse
obtained by replacing each \(-a_i\) by \(+a_i\). Thus \(T^{-1}\) also preserves \(\X\).
Let \(g\) be an indicator function of \(e_j\). Then \(g\circ T\) is
an indicator function of \(P\), 
and $\deg(g\circ T)=\deg(g)$. Therefore, ${\rm v}_P(I_\X)\leq {\rm
v}_{e_j}(I_\X)$. Applying the same argument to \(T^{-1}\), we obtain 
${\rm v}_{e_j}(I_\X)\leq {\rm v}_P(I_\X)$.
Hence, ${{\rm v}_P(I_\X)={\rm v}_{e_j}(I_\X)}$.
\end{proof}

The above result shows that if we want to determine the least degree of  
indicator functions of the points in $\X$, it is enough to consider the 
points $[e_i]$, where $i = 0, \ldots, n$. One may also deduce, from the 
above proof,  that 
the automorphism group of the code $C_\X(d)$ contains the transformations 
induced 
by the subgroup of nonsingular lower triangular matrices of the type which 
appears in it. The study of the full automorphism group of 
$C_\X(d)$ is of interest, but beyond the scope of this work. 
One may have an idea of the techniques that might be used in this study by
consulting \cite{berger}, where  
the automorphism group of the projective Reed-Muller code was determined.

\begin{lemma} \label{zero-function}
Let $P = (a_0 : \cdots : a_n) \in \X$ and let $j$ be the least index such 
that $a_j \neq 0$.
Let $g$
be a homogeneous polynomial of $S$ of degree $1\leq d\leq
\displaystyle \sum_{\substack{i = 0 \\ i \neq j}}^{n}(d_i-1)$.
If $g$ vanishes at all
points of $\mathcal{X}\setminus\{P\}$, then $g(P)=0$. 
\end{lemma}

\begin{proof}
We suppose, by contradiction,
that $g$ is not the zero function on $\X$, so $g \notin I_\X$, and let
$\X^* = K_0 \times \cdots \times K_n$. We write the point $P$ in
standard form with $a_j=1$. The vanishing ideal 
$I_{\X^*}$ of $\X^*$ is generated by the set 
$\{ t_i^{d_i}-t_i : i = 0, \ldots, n\}$ 
and its initial
ideal is given by ${\rm in}_\prec(I_{\X^*})=( \{ t_i^{d_i} : i=0, \ldots, 
n\})$.  
Notice that $V_{\mathcal{X}^*}(g)$, the zero set of $g$ in $\mathcal{X}^*$, 
has
cardinality equal to $|\mathcal{X}^*|-(d_j-1)$. By the division
algorithm, we can write 
$$
g=h_1(t_0^{d_0}-t_0)+\cdots+h_n(t_n^{d_n}-t_n)+F,
$$
where $F(0)=0$, $\deg(F)\leq\deg(g)$, and $\deg_{t_i}(F)\leq d_i-1$
for all $i$. Thus, one has 
$|V_{\mathcal{X}^*}(F)|=|V_{\mathcal{X}^*}(g)|$ and $F$ has exactly
$d_j-1$ nonzeros in $\X^*$. The leading monomial of $F$ has the form 
$t_0^{\gamma_0} \cdots t_n^{\gamma_n}$, where $\gamma_i\leq d_i-1$
for all $i$. Hence, by the footprint bound
\cite[p.~331]{rth-footprint}, $F$ has at least $\prod_{i = 0}^n (d_i
- \gamma_i)$ 
nonzeros in $\X^*$. Thus, $\prod_{i = 0}^n (d_i - \gamma_i) \leq  d_j - 1$. 

On the other hand, we may determine a lower bound for
$\prod_{i=0}^n(d_i - \gamma_i)$ using
\cite[Lemma~2.1]{car-2013}, which states that
\begin{equation*}
\min\bigg\{\prod_{i=0}^m(e_i-c_i) : c_i\in\mathbb{N},\, 
\sum_{i=0}^m c_i\leq e,\, 0 \leq c_i < e_i\; \forall\; i \bigg\} 
=(e_{k}-\ell)\prod_{i=k+1}^m e_i,
\end{equation*}
where $k,\ell\in\mathbb{N}$ are uniquely defined by 
the conditions $e = \sum_{i=0}^{k-1}(e_i-1) + \ell$, with $0\leq
\ell < e_{k}-1$, and we take $\prod_{i=k+1}^m e_i=1$ in the case $k = m$.
Since $0 \leq \gamma_i \leq d_i - 1$ for all $i=0, \ldots, n$ and 
$\sum_{i = 0}^n \gamma_i \leq \displaystyle \sum_{\substack{i = 0 \\ i \neq 
j}}^{n}(d_i-1)$, writing
\[
\sum_{\substack{i = 0 \\ i \neq j}}^{n}(d_i-1) = 
\sum_{i = 0}^{n - 1} (d_i - 1) \; + \; d_n - d_j 
\]
we get that, in this case,  $k = n$ and $\ell = d_n - d_j$ (because $d_j 
\geq 2$)
so that $\prod_{i=0}^n(d_i - \gamma_i) \geq d_n - (d_n - d_j) = d_j$, a 
contradiction.
Hence, $g(P)=0$ and
$g$ must be the zero 
function on $\X$.
\end{proof}

\begin{remark} 
The above result generalizes \cite[Lemma~5.4]{rth-footprint} because, for 
all $j \in \{0, \ldots, n\}$, we have 
\[
\sum_{i = 0}^{n - 1} (d_i - 1) \leq \sum_{\substack{i = 0 \\ i \neq 
j}}^{n}(d_i-1)
\]
since $d_0 \leq \cdots \leq d_n$.
\end{remark} 

In what follows we will need to identify monomials in $\Delta(I_\X)$. The 
footprint $\Delta(I_\X)$ is composed of two types of monomials: powers of 
the variables and monomials with more than one variable, where the power of 
the variable with the least index has no upper bound, while the powers of 
the other variables is limited. More explicitly, we have
\begin{align*}
\Delta(I_\X) &= \{ t_i^{\alpha_i} : \alpha_i \geq 0, i = 0, \ldots, n\} \\
 &\cup \{ t_{i_1}^{\alpha_{i_1}} \cdots t_{i_\ell}^{\alpha_{i_\ell}} : 
 \ell \geq 2, \alpha_{i_1} \geq 1, 1 \leq \alpha_{i_j} \leq d_{i_j} - 1 
 \textrm{ 
 for all } j = 2, \ldots, \ell \}.
\end{align*}

Let $f$ be the indicator function of $[e_j]$ of
Proposition~\ref{ind-funct}, $j=0,\ldots,n$. As is seen below, in this case $f$ has a 
simple expression.  
We observe that for $j=0$ and $j=n$ the monomials of $f$ are in
$\Delta(I_\X)$ while some monomials in $f$, for $j = 1, \ldots, n - 1$  are not in 
$\Delta(I_\X)$. 
For example, if $j = 1$ we have the monomial $t_0^{d_1 - 1} t_1^{1 + \sum_{i 
= 2}^n (d_i - 1)}$ which is not in $\Delta(I_\X)$.

The next result and its proof describe the monomials of $f$, 
and give constraints for the exponent of $t_j$ for 
any monomial in the remainder of $f$ on division by $\mathcal{G}$. 
The proof keeps track of the exponents of $t_j$
in the division algorithm. 
 
\begin{proposition} \label{powers-xj}
Let $f$ be the indicator function of $[e_j]$ which appears in
Proposition~\ref{ind-funct}, 
where $j=0,\ldots,n$, and let $h_j:={\rm NF}_\mathcal{G}(f)$ be the
remainder of $f$ on division by $\mathcal{G}$.  
Then, the power of $t_j$ in each monomial of $h_j$ is either $1$ or
an integer 
greater than $d_j-1$.
\end{proposition}
\begin{proof} We can write the indicator functions of Proposition
\ref{ind-funct},  
in the case where  $P = [e_j]$ for $j = 0, \ldots, n$, as follows: 
\[ 
f = t_0 \prod_{i = 1}^n  (t_i^{d_i - 1} -  t_0^{d_i - 1}) \;\;\; 
\textrm{if } j = 0; 
\] 
\[ 
f = t_1 (t_1^{d_1 - 1} - t_0^{d_1 - 1}) \prod_{i = 2}^n  
(t_i^{d_i - 1} -  t_1^{d_i - 1}) \;\;\; 
\textrm{if } j = 1; 
\] 
\[
f = t_j (f_1 - f_2) \; \prod_{\ell = j + 1}^n  
(t_\ell^{d_\ell - 1} -  t_j^{d_\ell - 1}) \;\;\; 
\textrm{if } 2 \leq j \leq n - 1, 
\]
where
\begin{equation} \label{f1}
\begin{split}
f_1 =& t_j^{m_j(d_j - 1)} - \sum_{i = 1}^{j - 1} t_i^{m_j(d_j - 1)}  \\ & + 
\sum_{s = 2}^{j - 
1} \sum_{1 \leq j_1 < \cdots < j_s \leq j - 1}
(-1)^s t_{j_1}^{m_j(d_j - 1) - \sum_{i = 2}^s (d_{j_i} \, - \, 1) } \; 
t_{j_2}^{d_{j_2}\, - 
\,1} \cdots t_{j_s}^{d_{j_s} \, - \, 1} ,
\end{split} 
\end{equation}
	
\[
f_2 = t_0^{m_j(d_j - 1) - \sum_{i = 1}^{j - 1} (d_{i} \, - \, 1)  } \prod_{i 
= 1}^{j - 1} 
(t_0^{d_{i} \, - \, 1} - t_i^{d_{i} \, - \, 1}),
\]
and
\[
f =  t_n (f_1 - f_2) \;\;\;  \textrm{if } j = n.
\]

Let $f$ be the indicator function for $[e_j]$ written above. 
If $j = n$ one easily checks that $h_n = t_n(f_1 - f_2)$ is a standard 
indicator function, and the power of $t_n$ in each monomial distinct from 
$t_n^{m_n(d_n - 1) + 1}$ is equal to 1, 
so the proposition holds in this case.

Assume now that $j \in \{0, \ldots, n - 1\}$.
We claim that 
the monomials of $\prod_{\ell = j + 1}^n  
(t_\ell^{d_\ell - 1} -  t_j^{d_\ell - 1})$ are of the form $t_j^{\alpha_j} 
\cdots  t_n^{\alpha_n}$ where:
\begin{equation}\label{apr6-26}
\left\{
\begin{array}{ll}
\alpha_j \in \{0, m(d_j - 1) \} & \textrm{ with } m > 0  \, ; \\
\alpha_\ell \in \{0, d_\ell - 1\} & \textrm{ if } \ell \in \{j+1, \ldots, 
n\} .
\end{array}
\right.
\end{equation}
The claim about $\alpha_\ell$, with $\ell \in \{j + 1, \ldots, n\}$ is 
clear. And, of course, the power of $t_j$ is zero 
in the monomial $t_{j + 1}^{d_{j + 1} - 1} \cdots t_n^{d_n - 1}$. From $K_j 
\subset K_\ell$  we get that $d_j - 1 \mid d_\ell - 1$, where $\ell \in \{j 
+ 1, \ldots, n\}$. Thus, $t_j$ appears in the monomials distinct from 
$t_{j + 1}^{d_{j + 1} - 1} \cdots t_n^{d_n - 1}$ with a power of the form $m 
(d_j - 1)$, where $m > 0$ depends on the monomial. This finishes the proof 
of the claim.

Inspecting $h_0 = t_0 \prod_{i = 1}^n  (t_i^{d_i - 1} -  t_0^{d_i -
1})$, we 
see that this is a standard indicator function and, 
except for the monomial $t_0 \prod_{i = 1}^n  t_i^{d_i - 1}$, 
the power of $t_0$ in 
a monomial of $h_0$ is of the form $m(d_0 - 1) + 1$, where $m > 0 $ depends on the 
monomial,  so 
the proposition also holds for $j = 0$.

Now, for $j \in \{1, \ldots, n - 1\}$,  we want to determine the monomials 
which appear in the remainder of the 
division of $t_j f_1 \prod_{\ell = j + 1}^n  (t_\ell^{d_\ell - 1} -  
t_j^{d_\ell - 1})$ by the polynomials in 
\[
\mathcal{G} := \left\{ t_i t_j (t_j^{d_j - 1} - t_i^{d_j - 1}), \;  i < j ; 
i , j = 0,\ldots,n 
\right\}
\]
where $f_1$ is the polynomial in Eq.~\eqref{f1}. Looking at the product 
of 
monomials $t_j^{\alpha_j} 
\cdots  t_n^{\alpha_n}$ that appear in  
$\prod_{\ell = j + 1}^n  
(t_\ell^{d_\ell - 1} -  t_j^{d_\ell - 1})$, which were described in 
Eq.~\eqref{apr6-26}, 
with monomials which appear in $t_j f_1$, we have:\\
a) if $\alpha_j=0$, the product of $t_{j + 1}^{d_{j + 1} - 1} \cdots 
t_n^{d_n - 1}$
and the monomials of $t_jf_1$ yield only monomials which are in 
$\Delta(I_\X)$. The power of $t_j$ in these monomials is always equal to 1, 
except for the monomial $t_j^{m_j(d_j - 1) + 1}  t_{j + 1}^{d_{j + 1} - 1} 
\cdots t_n^{d_n - 1}$; \\
b) if $\alpha_j = m (d_j - 1)$, the product of a monomial $t_j^{\alpha_j} 
\cdots  t_n^{\alpha_n}$ with  $t_j 
t_j^{m_j (d_j 
- 1)}$ is equal to $t_j^{(m + m_j)(d_j - 1) + 1} t_{j + 1}^{\alpha_{j+1}}
\cdots  t_n^{\alpha_n}$ and it is in $\Delta(I_\X)$; \\
c) if $\alpha_j = m (d_j - 1)$, the product of a monomial $t_j^{\alpha_j} 
\cdots  t_n^{\alpha_n}$ with  $t_j 
t_i^{m_j(d_j - 1)}$, where $i \in \{1, \ldots, j - 1\}$ is equal to 
$t_j^{m(d_j - 1) + 1} t_i^{m_j(d_j - 1)}t_{j+1}^{\alpha_{j+1}}\cdots
t_n^{\alpha_n}$. 
This monomial is not in 
$\Delta(I_\X)$ and its remainder in the division by $\mathcal{G}$ (or, more 
precisely, by $t_i t_j^{d_j} - t_i^{d_j} t_j$) is equal to 
$$t_i^{(m + m_j)(d_j - 1)} t_jt_{j+1}^{\alpha_{j+1}}\cdots
t_n^{\alpha_n},$$
we arrive at this by repeated application 
of divisions like 
\begin{align*}
t_j^{m(d_j - 1) + 1} t_i^{m_j(d_j - 1)} &= t_j^{(m - 1)(d_j - 1)} 
t_i^{m_j(d_j - 1) - 1} (t_i t_j^{d_j} - t_i^{d_j} t_j) \\ &+ 
t_j^{(m - 1)(d_j - 1) + 1} t_i^{(m_j+ 1)(d_j - 1)}
; 
\end{align*}
d) the product of a monomial $t_j^{\alpha_j} 
\cdots  t_n^{\alpha_n}$,  where $\alpha_j = m (d_j - 1)$, with  
\[
t_j t_{j_1}^{m_j(d_j - 1) - \sum_{i = 2}^s (d_{j_i} \, - \, 1) } \; 
t_{j_2}^{d_{j_2}\, - 
\,1} \cdots t_{j_s}^{d_{j_s} \, - \, 1},
\]
 where
$1 \leq j_1 < \cdots < j_s \leq j - 1$ and $s \in \{2, \ldots, j - 1\}$, 
is equal to 
\[
 t_{j_1}^{m_j(d_j - 1) - \sum_{i = 2}^s (d_{j_i} \, - \, 1) } \; 
t_{j_2}^{d_{j_2}\, - \,1} \cdots t_{j_s}^{d_{j_s} \, - \, 1} t_j^{m(d_j - 1) 
+ 1} 
t_{j+1}^{\alpha_{j+1}} \cdots t_n^{\alpha_n}.
\]
This monomial is not in 
$\Delta(I_\X)$ and, as above, it is easy to check that its remainder in the 
division by $\mathcal{G}$ (or, more 
precisely, by $t_{j_1} t_j^{d_j} - t_{j_1}^{d_j} t_j$) is
\[
 t_{j_1}^{(m_j + m)(d_j - 1) - \sum_{i = 2}^s (d_{j_i} \, - \, 1) } \; 
t_{j_2}^{d_{j_2}\, - \,1} \cdots t_{j_s}^{d_{j_s} \, - \, 1} t_j  
t_{j+1}^{\alpha_{j+1}} \cdots t_n^{\alpha_n}.
\]
Thus, in all cases we have that the power of $t_j$ in the monomials of the 
remainder is either 1 or greater than $d_j - 1$.

In the same way, we now prove that all the monomials which appear in the 
remainder of the division of  $t_j f_2 \prod_{\ell = j + 1}^n  
(t_\ell^{d_\ell - 1} -  t_j^{d_\ell - 1})$ by $\mathcal{G}$ are multiples of 
$t_j$ but not of $t_j^2$.

The monomials of 
$\prod_{\ell = j + 1}^n  
(t_\ell^{d_\ell - 1} -  t_j^{d_\ell - 1})$, which are described in 
Eq.~\eqref{apr6-26}, can be written as 
$t_j^{\alpha_j}\cdots t_n^{\alpha_n}$, where
$\alpha_j\in\{0,m(d_j-1)\}$ with  $m>0$ and
$\alpha_\ell\in\{0,d_\ell-1\}$ for $\ell=j+1,\ldots,n$.  

A) Assume that $\alpha_j=0$. The product of $t_{j+1}^{\alpha_{j+1}}\cdots 
t_n^{\alpha_n}$
and the monomials of $t_jf_2$ yield only monomials which are in 
$\Delta(I_\X)$, and the power of $t_j$ in these monomials is always equal
to $1$. 

B) Assume that $\alpha_j = m (d_j - 1)$ with $m>0$. The monomials of
$t_jf_2$ are of the form
\[
t_j t_0^{m_j(d_j - 1) - \sum_{i =0}^{j-1}n_i(d_i-1)} \; 
t_1^{\alpha_1} \cdots t_{j-1}^{\alpha_{j-1}},
\] 
where $n_i\in\{0,1\}$ for $i=0,\ldots,j-1$ and $\alpha_i\leq d_i-1$
for $i=1,\ldots,j-1$. Thus the product of a monomial 
$t_j^{\alpha_j}\cdots t_n^{\alpha_n}$ of $\prod_{\ell = j + 1}^n  
(t_\ell^{d_\ell - 1} -  t_j^{d_\ell - 1})$ and a monomial of $t_jf_2$
has the form
\[
t_j t_0^{m_j(d_j - 1) - \sum_{i =0}^{j-1}n_i(d_i-1)} \; 
t_1^{\alpha_1} \cdots t_{j-1}^{\alpha_{j-1}}t_j^{\alpha_j}
t_{j+1}^{\alpha_{j+1}}\cdots t_n^{\alpha_n}.
\] 
\quad This monomial is not in 
$\Delta(I_\X)$ and its remainder in the division by $\mathcal{G}$ (or, more 
precisely, by $t_0 t_j^{d_j} - t_0^{d_j} t_j$) is equal to
\[
t_0^{(m_j+m)(d_j - 1) - \sum_{i =0}^{j-1}n_i(d_i-1)} \; 
t_1^{\alpha_1} \cdots t_{j-1}^{\alpha_{j-1}}t_j
t_{j+1}^{\alpha_{j+1}}\cdots t_n^{\alpha_n},
\] 
we obtain this by a repeated application 
of divisions like 
\begin{align*}
&t_0^{m_j(d_j - 1) - \sum_{i =0}^{j-1}n_i(d_i-1)} \; 
t_j^{m(d_j-1)+1}=\\
& (t_0 t_j^{d_j} - t_0^{d_j} t_j)t_0^{m_j(d_j - 1) - \sum_{i
=0}^{j-1}n_i(d_i-1)-1} \; 
t_j^{(m-1)(d_j-1)}   \\ &+ 
t_0^{(m_j+1)(d_j - 1) - \sum_{i
=0}^{j-1}n_i(d_i-1)} \; 
t_j^{(m-1)(d_j-1)+1}.
\end{align*}

Thus, in both cases the power of $t_j$ in the monomials of the 
remainder is $1$.
\end{proof}

\begin{lemma}\label{feb24-26-bis} If $t^\beta=t_k^{\beta_k}\cdots 
t_n^{\beta_n}$,
$\beta_k\geq 1$, is a standard monomial of $I_\X$ such that $t_j$
divides $t^\beta$ and 
$t_j t^{\beta}\notin\Delta(I_{\X})$, then $k<j$ and $\beta_j=d_j-1$. 
\end{lemma}

\begin{proof} As $t_k t^\beta\in\Delta(I_\X)$, one has $k<j$.
Therefore, $\beta_j\leq d_j-1$ since $t^\beta\in\Delta(I_X)$ and 
$t_kt_j^{d_j}-t_k^{d_j}t_j\in\mathcal{G}$. If $\beta_j<d_j-1$, then
$\beta_j+1\leq d_j-1$ and
$t_j t^\beta\in \Delta(I_\X)$, a contradiction. Thus,
$\beta_j=d_j-1$. 
\end{proof}

\begin{lemma}\label{feb24-26-1-bis} Let $P = (a_0 : \cdots : a_n) \in
\X$, let $g$ be a standard indicator 
function for $P$, let 
$j$ be the least integer such 
that $a_j \neq 0$, let $t^\alpha$ be any of the standard monomials that 
occur in $g$,
and let ${\rm NF}_{\mathcal{G}}(t_jg)$ be the remainder of $t_jg$ on
division by $I_\X$. 
 The following hold. 
\begin{enumerate}
\item[\rm(a)] $t_j$ divides $g$.
\item[\rm(b)] If $t^\alpha=t_0^{\alpha_0}\cdots t_n^{\alpha_n}$ and 
$\alpha_j\leq d_j-2$, then
$t_j t^{\alpha}\in\Delta(I_{\X})$ and $t_j t^\alpha$ appears in ${\rm
NF}_{\mathcal{G}}(t_jg)$.
\end{enumerate}
\end{lemma}

\begin{proof} (a) By Proposition~\ref{xj-div-f}, we get that $t_j$ divides 
$g$.

(b) Let $t^\beta=t_k^{\beta_k}\cdots t_n^{\beta_n}$,
$\beta_k\geq 1$, be any monomial that occur in $g$ such that 
$t_j t^{\beta}\notin\Delta(I_{\X})$. Then, by part (a) and 
Lemma~\ref{feb24-26-bis}, 
$k<j$ and $\beta_j=d_j-1$. It suffices to show the equality 
$$ 
{\rm 
NF}_\mathcal{G}(t_j t^\beta)=t_k^{d_j+\beta_k-1}t_{k+1}^{\beta_{k+1}}\cdots
t_{j-1}^{\beta_{j-1}}t_jt_{j+1}^{\beta_{j+1}}\cdots t_{n}^{\beta_{n}}
$$
because this proves that the standard monomial
$t_j t^{\alpha}$ will not appear as a remainder at any step of the 
division of $t_jg$ by $\mathcal{G}$ since $t_j^2$ divides $t_jg$  by
part (a), and consequently $t_j^2$ divides $t_j t^\alpha$ and $t_j t^\alpha$ 
appears in ${\rm
NF}_{\mathcal{G}}(t_jg)$. The equality follow from
\begin{align*}
&t_jt^\beta-(t_kt_j^{d_j}-t_k^{d_j}t_j)(t_k^{\beta_k-1}t_{k+1}^{\beta_{k+1}}\cdots
t_{j-1}^{\beta_{j-1}}t_{j+1}^{\beta_{j+1}}\cdots t_{n}^{\beta_{n}})\\
&=
t_k^{d_j+\beta_k-1}t_{k+1}^{\beta_{k+1}}\cdots
t_{j-1}^{\beta_{j-1}}t_jt_{j+1}^{\beta_{j+1}}\cdots t_{n}^{\beta_{n}}
\end{align*}
and noticing that $t_k^{d_j+\beta_k-1}t_{k+1}^{\beta_{k+1}}\cdots
t_{j-1}^{\beta_{j-1}}t_jt_{j+1}^{\beta_{j+1}}\cdots 
t_{n}^{\beta_{n}}\in\Delta(I_\X)$. 
\end{proof}

\begin{theorem}\label{cicero-maria-vila-ei}
Let $j \in \{0, \ldots, n\}$ and let $P = [e_j] \in \X$. If $f$ is the 
indicator of $P$ that 
appears in Proposition \ref{ind-funct}, then $f$ 
has minimal degree among all
indicator functions of $P$ and 
$$
{\rm v}_P(I_\X)=\deg(f)=\begin{cases}
{\rm reg}(H_\X)=1 + \sum_{i = 1}^n (d_i - 1)&\mbox{if } j\in\{0,1\},\\
 m_j (d_j - 1) + 1 + \sum_{i = j + 1}^n (d_i - 1)&\mbox{if }j \in
 \{2, \ldots, n-1\},\\
m_n(d_n-1)+1&\mbox{if }j=n,
\end{cases}
$$
where $m_j$ is the least integer $m$ such that 
$m(d_j - 1) > \sum_{i = 1}^{j - 1} (d_i - 1)$.
\end{theorem}

\begin{proof} 
Let $h$ be an indicator function of $P$ of degree $d$. We argue by
contradiction assuming that $\deg(h)<\deg(f)$. We may assume that 
$d$ is equal to $\deg(f)-1$ by multiplying $h$ by $t_j^{\deg(f)-1-d}$.
By the division algorithm, there is a standard indicator function $g$ of
$P$ of degree $d$, and the remainder ${\rm NF}_\mathcal{G}(t_jg)$ of
$t_jg$ on the division by $\mathcal{G}$ is a standard indicator function
of degree $d+1$. Then, by the uniqueness of standard indicators
\cite[Lemma~3.2(e)]{villa2024} and noticing that $f$ is an
indicator function of degree $d+1$, one has that 
${\rm NF}_\mathcal{G}(t_jg) = \mu {\rm NF}_\mathcal{G}(f)$ for some $\mu \in 
K^*$. 

From Proposition \ref{xj-div-f} we know that each monomial of the (standard)
indicator function $g$ is a multiple of $t_j$. If the power of $t_j$ in a 
monomial $M$ of $g$ is in the set $\{1, \ldots, d_j - 2\}$, then 
$t_j M \in \Delta(I_\X)$. Hence, by Lemma~\ref{feb24-26-1-bis}, $t_j
M$ is a monomial 
that appears in 
${\rm NF}_\mathcal{G}(f)$ where the power of $t_j$ in $t_jM$ is in the set 
$\{2, 
\ldots, d_j - 1\}$. Yet, Proposition \ref{powers-xj} says there's no such 
monomial in ${\rm NF}_\mathcal{G}(f)$. Therefore, the power of $t_j$
in each monomial of $g$ is at least $d_j-1$ and we can write
$g=t_j^{d_j-1}H$ for some polynomial $H$ in $S$. 
Clearly $t_j H$ is an indicator function for $[e_j]$. 
When $j = 0$ we have $\deg(f) = 1 + \sum_{i = 1}^n (d_i - 1)$ so that
$$\deg(t_0 H) = 1 + \deg(g) - (d_0 - 1) = 1 + \sum_{i = 1}^n (d_i - 1) - 
(d_0 - 1) \leq \sum_{i = 1}^n (d_i - 1).$$
 When $j \geq 1$, taking 
$m_1 := 1$ we have $\deg(f) =  m_j (d_j - 1) + 1 + \sum_{i = j + 1}^n (d_i - 
1)$, and  using the minimality of $m_j$ (or $m_1 = 1$), we obtain: 
\begin{align*} 
\deg&(t_j H) = 1 + \deg(g) - (d_j - 1) = m_j(d_j - 1) + \sum_{i = j + 1}^n 
(d_i 
- 1) 
- (d_j - 1) + 1 \\
&= (m_j - 1) (d_j - 1) + \sum_{i = j + 1}^n (d_i - 1) + 1 \leq \sum_{i = 
1}^{j - 1} (d_i - 1) + \sum_{i = j + 1}^n (d_i - 1) + 1 \\
&\leq \sum_{\substack{i = 0 \\ i \neq j}}^{n}(d_i-1),
\end{align*}
where 
the
minimality of $m_j$ is used in the first inequality. 
So for all $j \in \{0, \ldots, n\}$ we have 
\[
\deg(t_j H) \leq \sum_{\substack{i = 0 \\ i \neq j}}^{n}(d_i-1).
\]
Since $t_j H$ vanishes on $\X \setminus \{ [e_j]\}$, from Lemma 
\ref{zero-function}, we have $t_j H([e_j]) = 0$, a contradiction.
This proves that no indicator function for $[e_j]$ can have a degree less 
than $\deg(f)$. By \cite[Lemma~3.2(b)]{villa2024}, ${\rm
v}_{P}(I_\X)$  is the least degree of an indicator function of $P$.
Thus, ${\rm v}_P(I_\X)=\deg(f)$. The formula for $\deg(f)$ follows
from Proposition~\ref{ind-funct}.
\end{proof}

\begin{theorem}\label{cicero-maria-vila-reg} 
Let $\X=[K_0\times\cdots\times K_n]\subset\mathbb{P}^n$ be a
projective nested product of fields, let $d_i=|K_i|$, and let 
${\rm reg}(\delta_\X)$ be the regularity index of $\delta_X$. The following
hold.
\begin{enumerate}
\item[\rm(a)] ${\rm reg}(\delta_\X)=m_n(d_n - 1)+1$, where
$m_n=\min\{m : m\in\mathbb{N},\ m(d_n-1)>\sum_{i=1}^{n-1}(d_i-1)\}$. 
\item[\rm(b)] ${\rm v}_P(I_\X)\in\{{\rm v}_{e_0}(I_\X),\ldots,{\rm
v}_{e_n}(I_\X)\}$
for all $P\in \X$, where $e_j$ is the $j$-th unit vector of
$\mathbb{P}^n$.
\item[\rm(c)] $1+\sum_{i=1}^n(d_i-1)={\rm reg}(H_\X)={\rm
v}_{e_0}(I_\X)\geq\cdots\geq{\rm
v}_{e_n}(I_\X)=m_n(d_n-1)+1={\rm reg}(\delta_\X).
$
\end{enumerate}
\end{theorem}
\begin{proof}
(a) For each $P\in\X$, let $f_P$ be the indicator function that
appears in Proposition~\ref{ind-funct}. By Lemma~\ref{min-deg}, one
has 
$$
\min\{\deg(f_P): P\in \X\}=m_n(d_n-1)+1,
$$
and by Theorems~\ref{p-ej} and \ref{cicero-maria-vila-ei}, one has
\begin{equation}\label{apr16-26}
{\rm v}_P(I_\X)={\rm v}_{e_j}(I_\X)\ \mbox{ and }\ {\rm
v}_{e_j}(I_\X)=\deg(f_{e_j})=\deg(f_P)
\end{equation}
where $j \in \{0, \ldots, n\}$ is the least integer such 
that the $j$-th entry of $P$ is nonzero. Therefore, 
\begin{equation}\label{apr16-26-2}
{\rm reg}(\delta_\X)={\rm v}(I_\X)=\min\{{\rm v}_P(I_\X): P\in\X\}
=\min\{\deg(f_P): P\in\X\}=m_n(d_n-1)+1.
\end{equation}
\quad (b) This follows from Theorem~\ref{p-ej}, see
Eq.~\eqref{apr16-26}. 

(c) The first equality follows from Lemma~\ref{lemma1.1}. By
Eq.~\eqref{apr16-26}, we get ${\rm v}_{e_j}(I_\X)=\deg(f_{e_j})$ for 
$j=0,\ldots,n$. Then, the
inequalities ${\rm
v}_{e_0}(I_\X)\geq\cdots\geq{\rm
v}_{e_n}(I_\X)$ follow from Eq.~\eqref{apr16-26-1}, in the proof of
Lemma~\ref{min-deg}, because Eq.~\eqref{apr16-26-1} tells us that
$\deg(f_{e_{j+1}})\leq\deg(f_{e_j})$ for $j=0,\ldots,n-1$. By 
Eq.~\eqref{apr16-26}, we get
$$
\max\{{\rm v}_P(I_\X): P\in\X\}
=\max\{\deg(f_{e_j}): j=0,\ldots,n\}={\rm v}_{e_0}(I_\X).
$$ 
\quad From \cite[Lemma~3.2(d)]{villa2024}, one has that $\max\{{\rm
v}_P(I_\X): P\in \X\}={\rm reg}(H_\X)$. Thus, the second equality holds. The
last two equalities follow from Eqs.~\eqref{apr16-26} and \eqref{apr16-26-2}.   
\end{proof}

\begin{corollary}\label{coro1} 
Let $\X=[K_0\times\cdots\times K_n]\subset\mathbb{P}^n$ be a
projective nested product of fields, let $d_i=|K_i|$. The
following hold. 
\begin{enumerate}
\item[\rm(a)] $\X$ is Cayley--Bacharach if and only if 
$\sum_{i=1}^n(d_i-1)+1=m_n(d_n-1)+1$. 
\item[\rm(b)] \cite[Theorem~1]{sorensen} $\X=\mathbb{P}^n$ is
Cayley--Bacharach and ${\rm reg}(\delta_\X)={\rm reg}(H_\X)=n(q-1)+1$.
\end{enumerate}
\end{corollary}

\begin{proof}
(a) Since ${\rm reg}(\delta_\X)=\min\{{\rm v}_P(I_\X) : P\in \X\}$ 
and ${\rm reg}(H_\X)=\max\{{\rm
v}_P(I_\X) : P\in \X\}$, one has that $\X$ is Cayley--Bacharach 
if and only if ${\rm reg}(\delta_\X)={\rm reg}(H_\X)$, and the result
follows from Theorem~\ref{cicero-maria-vila-reg}(c).

(b) As $d_i=q$ for all $i$ and 
$m_n=\min\{m : m\in\mathbb{N},\ m(q-1)>(q-1)(n-1)\}=n$, 
by part (a) and its proof, $\X$ is Cayley--Bacharach and ${\rm
reg}(\delta_\X)={\rm reg}(H_\X)=n(q-1)+1$.  
\end{proof}

\section{Examples}

\begin{example}\label{ex2}
Let $K_0\subset K_1\subset K_5=K$ be a tower of subfields of a finite
field $K=\mathbb{F}_q$ and let $\X=[K_0\times\cdots\times K_5]$ be a
projective nested product 
of fields with $d_i=|K_i|$, $i=0,\ldots,5$. 
For convenience of notation we say that $\X$ is the 
projective nested product 
of fields with defining sequence
$\mathbf{d}=(d_0,d_1,d_2,d_3,d_4,d_5)$ and denote the local v-number
${\rm v}_{e_j}(I_\X)$ by ${\rm v}_{e_j}$ for
$j=0,\ldots,5$. Using Eq.~\eqref{apr19-26} and
Theorems~\ref{cicero-maria-vila-ei}--\ref{cicero-maria-vila-reg}, we obtain Table~\ref{vei}
that in particular gives all possible values of ${\rm v}_P(I_\X)$ for $P\in\X$.
One can use Procedure~\ref{procedure1} for \textit{Macaulay}$2$ \cite{mac2} 
to verify the values of this table.
\begin{table}[H] 
\hspace{5cm}
\begin{center}
\begin{tabular}{c|c|c|c|c|c|c|c|c|c}
$(d_0,d_1,d_2,d_3,d_4,d_5)$& ${\rm v}_{e_0}$& ${\rm
v}_{e_1}$ & ${\rm v}_{e_2}$ & ${\rm v}_{e_3}$& ${\rm
v}_{e_4}$ & ${\rm v}_{e_5}$ & $|\X|$ & ${\rm reg}(H_\X)$ & ${\rm
reg}(\delta_\X)$
\\ \hline $(2,2,4,4,16,16)$ & 38 & 38 & 37 & 37 &  31&
31&13585& 38 & 31
\\ \hline $(2,2,2,4,4,4)$ & 12 & 12& 12 & 10 &  10&10& 469 & 12 & 10
\\ \hline $(2,2,2,2,4,4)$ & 10 & 10& 10 & 10 &  10&10&245 & 10 & 10
\\ \hline $(2,2,2,2,2,4)$ & 8 & 8& 8 & 8 &  8&7&125 & 8 &7
\\ \hline \vspace{-4mm} 
\\ 
$(2,2,2^2,2^4,2^8,2^8)$ & 530 & 530& 529& 526 &  511&511& 13697281&
530 & 511 
\\ \hline \vspace{-4mm} 
\\ 
$(3,3,3,3,3^2,3^4)$ & 95 & 95& 95& 95 &  89&81& 29242&
95 & 81 
\end{tabular}
\end{center}
\caption{All possible values of ${\rm v}_P(I_\X)$ for $P\in\X$.}\label{vei} 
\end{table}
\end{example}

\begin{example}\label{ex1}
We recall an example devised by R.\ Villarreal which appeared in
\cite[Example~3.2]{cnl2017}.
We let $K = \mathbb{F}_4$, $K_0 = K_1 = \mathbb{F}_2$, $K_2 = 
\mathbb{F}_4$, $d_0=d_1=2$, $d_2=4$, and $\X=[K_0\times K_1\times
K_2]$. Table~\ref{mdf} gives the values of the minimum distance 
$\delta_{\mathcal{X}}(d)$ of $C_{\mathcal{X}}(d)$, as $d$ varies from
$1$ to $4$. 		
\begin{table}[H]
\begin{center}
\begin{tabular}{c|c|c|c|c}
$d$ & 1 & 2 & 3 & 4 \\ \hline 
$\delta_{\mathcal{X}}(d)$ & 8 & 4 & 3 & 1 
\end{tabular}
\end{center}
\caption{The minimum distance function $\delta_\X$.}\label{mdf}
\end{table}

Thus, $\mathrm{reg}(\delta_{\mathcal{X}}) = 4$, while the regularity
of $H_\mathcal{X}$, by Lemma \ref{lemma1.1}, is  $\sum_{i = 1 }^2 (d_i - 1) + 1 = 
5$, so this is an example 
where $\mathrm{reg}(\delta_{\mathcal{X}})  <
\mathrm{reg}(H_\mathcal{X})$. Furthermore, the defining sequence of
$\X$ is $(2,2,4)$ and the local v-numbers are given in
Table~\ref{vila-ex}. 
\begin{table}[H]
\begin{center}
\begin{tabular}{c|c|c|c|c|c|c}
$(d_0,d_1,d_2)$&${\rm v}_{e_0}$ & ${\rm v}_{e_1}$ & ${\rm v}_{e_2}$ &
$|\X|$ & ${\rm reg}(H_\X)$ & ${\rm 
reg}(\delta_\X)$
\\ \hline 
$(2,2,4)$ & 5& 5 & 4 & 13 &5 & 4
\end{tabular}
\end{center}
\caption{All possible local v-numbers of points in $\X$.}\label{vila-ex}
\end{table}

Note that the set of leading monomials
of the Gr\"obner 
basis 
$$\mathcal{G}
=\{t_1^2t_0+t_1t_0^2,t_2^4t_0+t_2t_0^4,t_2^4t_1+t_2t_1^4\}
$$ 
given in Eq.~\eqref{basedeG} is $\{ t_0 t_1^2, t_0 t_2^4,$ $t_1 t_2^4\}$. 
Thus, a basis for $\mathbb{F}_4[t_0, t_1, t_2]_4/I(\mathcal{X})_4$ is 
\begin{equation*} 
\begin{split} 
\{t_0^4, t_0^3 t_1, t_0^3 t_2, t_0^2 t_2^2, &t_0 t_2^3, t_1^4, t_1^3 t_2, \\ 
&t_1^2 
t_2^2, t_1 t_2^3, t_2^4, t_0 t_1 t_2^2, t_0^2 t_1 t_2\}
\end{split} 
\end{equation*}
which has 12 elements and a basis 
for $\mathbb{F}_d[t_0, t_1, t_2]_d/I(\mathcal{X})_d$, when $d \geq 5$ always 
has 13 elements because ${\rm reg}(H_\X)=5$ and consequently
$H_{\X}(d)=|\X|=13$ for all $d\geq 5$. 

Let $a$ be a generator of the multiplicative cyclic group
$\mathbb{F}_4^*=\mathbb{F}_{4}\setminus\{0\}$ of $\mathbb{F}_4$.
Using \textit{Macaulay}$2$ \cite{mac2} and adapting Procedure~\ref{procedure1}, we obtain that the unique standard
indicator function $f_{e_j}$ of $e_j$ of degree ${\rm v}_{e_j}(I_\X)$
for $j=0,1,2$ is given by 
\begin{align*}
f_{e_0}&=  (t_0)(t_1+t_0)(t_2+t_0)(t_2+at_0)(t_2+(a+1)t_0)\\ 
       &=t_2^3t_1t_0+t_2^3t_0^2+t_1t_0^4+t_0^5; \\
f_{e_1}&=
(t_1)(t_1+t_0)(t_2+t_1+t_0)(t_2+at_1+at_0)(t_2+(a+1)t_1+(a+1)t_0)\\ 
       &=t_2^3t_1^2+t_2^3t_1t_0+t_1^5+t_1t_0^4;\\
f_{e_2}&= (t_2)(t_2^3+t_1^3+t_1t_0^2+t_0^3)\\
       &=t_2^4+t_2t_1^3+t_2t_1t_0^2+t_2t_0^3. 
\end{align*}
\quad Note that $f_{e_0}$ and $f_{e_2}$ are the indicator functions 
of Proposition~\ref{ind-funct} that correspond to $e_0$ and $e_2$. The
indicator function of Proposition~\ref{ind-funct} that
correspond to $e_1$ is 
$$f=t_1(t_1-t_0)(t_2^3-t_1^3)=t_1^2t_2^3-t_1^5-t_0t_1t_2^3+t_0t_1^4
$$ 
which is not standard since $t_0t_1^4$ appears in $f$. The remainder
of $t_0t_1^4$ on division by the binomial
$t_0t_1^2-t_0^2t_1\in\mathcal{G}$ is $t_0^4t_1$. Thus, ${\rm
NF}_\mathcal{G}(f)=f_{e_1}$.

\end{example}

\begin{appendix}

\section{Finding all possible minimal degrees of indicator 
functions}\label{procedure-degrees}
We give a procedure for \textit{Macaulay}$2$
\cite{mac2}, based on Theorem~\ref{cicero-maria-vila-reg}, to
determine all possible values of ${\rm v}_P(I_\X)$ for $P\in\X$. 

\begin{procedure}\label{procedure1}\rm 
The following procedure for \textit{Macaulay}$2$ \cite{mac2}
computes all possible values of the local v-numbers ${\rm v}_P(I_\X)$ for $P\in\X$. This
procedure determines the unique standard indicator function
$f_{e_j}$ of degree ${\rm v}_{e_j}(I_\X)$ for each
$e_j\in\mathbb{P}^n$, where $e_j$ is the $j$-th unit vector.
The input is the sequence $\mathbf{d}=(d_0,\ldots,d_n)$ defining $\X$,
the finite field $\mathbb{F}_q$, and the vanishing ideals $I_{e_j}$
of $e_j$ for $j=0,\ldots,n$. 
This procedure corresponds to the sequence $(2,2,2^2,2^4,2^8,2^8)$ and
the finite field $K=\mathbb{F}_{2^8}$
of Table~\ref{vei} in Example~\ref{ex2}. 
\begin{verbatim}
restart
p=2,s=8, q=p^s
--Finite field
K=GF(q,Variable=>a)
S=K[t5,t4,t3,t2,t1,t0,MonomialOrder=>GLex]
--Defining sequence of X
d0=2,d1=2,d2=2^2,d3=2^4,d4=2^8,d5=2^8
--Vanishing ideal I of X
I=ideal(t0*t1^d1-t0^d1*t1, t0*t2^d2-t0^d2*t2, t0*t3^d3-t0^d3*t3,
t0*t4^d4-t0^d4*t4, t0*t5^d5-t0^d5*t5, t1*t2^d2-t1^d2*t2,
t1*t3^d3-t1^d3*t3, t1*t4^d4-t1^d4*t4, t1*t5^d5-t1^d5*t5,
t2*t3^d3-t2^d3*t3, t2*t4^d4-t2^d4*t4, t2*t5^d5-t2^d5*t5,
t3*t4^d4-t3^d4*t4, t3*t5^d5-t3^d5*t5, t4*t5^d5-t4^d5*t5)
--Number of elements of X
degree I
--Regularity of the Hilbert function of X
regularity coker gens gb ideal(leadTerm(I))
--The list of vanishing ideals of e0,...,e5
AssI={ideal(t1,t2,t3,t4,t5),ideal(t0,t2,t3,t4,t5),
ideal(t0,t1,t3,t4,t5), ideal(t0,t1,t2,t4,t5),
ideal(t0,t1,t2,t3,t5), ideal(t0,t1,t2,t3,t4)}
--The degree of the unique standard indicator function 
--of degree v_{e_j}(I) 
F=(j)-> degrees mingens (quotient(I,AssI#j)/I)-set{{{0}}}
--The unique standard indicator function of e_j
--of degree v_{e_j}(I)
F1=(j)->(flatten entries mingens (quotient(I,AssI#j)/I))#0
--The list of the standard indicator functions of 
--minimal degree of e0,...,e5, that is, 
--the list v_{e_0}(I),...,v_{e_5}(I)
apply(0..#AssI-1,F1)
--The regularity of the minimum distance function, that is, 
--the v-number of the vanishing ideal of X
vnumber=flatten flatten  min apply(0..#AssI-1,F)
--The list of all possible degrees of the 
--indicator functions of minimal degree 
localvnumbers=flatten flatten apply(0..#AssI-1,F)
\end{verbatim}
\end{procedure}
\end{appendix}

\section*{Acknowledgments.} 
\textit{Macaulay}$2$ \cite{mac2} was used to implement algorithms for
computing all possible values of local v-numbers and standard
indicators functions.






\end{document}